\documentclass[10pt,leqno]{amsart}
\usepackage{indentfirst, amssymb,amsmath,amsthm}
\newtheorem*{theoA}{Theorem A}
\newtheorem*{theoB}{Theorem B}
\newtheorem*{theoC}{Theorem C}

\newtheorem{theo}{Theorem}[section]
\newtheorem{lem}{Lemma}[section]
\newtheorem{conj}{Conjecture}[section]

\newtheorem{exm}{Example}[section]
\newtheorem{defi}{Definition}[section]
\newtheorem{rem}{Remark}[section]
\newtheorem{ques}{Question}[section]
\newcommand{\ol}{\overline}
\newcommand{\be}{\begin{equation}}
\newcommand{\ee}{\end{equation}}
\newcommand{\beas}{\begin{eqnarray*}}
\newcommand{\eeas}{\end{eqnarray*}}
\newcommand{\bea}{\begin{eqnarray}}
\newcommand{\eea}{\end{eqnarray}}

\numberwithin{equation}{section}
\begin{document}
\title[Set  shared by an entire function with its $k$-th derivatives ...]{Set  shared by an entire function with its $k$-th derivatives using normal families}
\date{}
\author[M. B. Ahamed ]{ Molla Basir Ahamed }
\date{}
\address{ Department of Mathematics, Kalipada Ghosh Tarai Mahavidyalaya, West Bengal 734014, India.}
\email{bsrhmd116@gmail.com, bsrhmd117@gmail.com.}
\maketitle
\let\thefootnote\relax
\footnotetext{Type set by \AmS -\LaTeX}
\setcounter{footnote}{0}

\begin{abstract} In this paper, we study a problem of non-constant entire function $ f $  that shares a set $ \mathcal{S}=\{a,b,c\} $ with its $ k $-th derivative $ f^{(k)} $, where $ a, b $ and $ c $ are any three distinct complex numbers. We have found a gap in the statement of the main result of \textit{Chang-Fang-Zalcman} \cite{Cha & Fan & al-ADM-2007} and with some help of the method used by \textit{Chang-Fang-Zalcman}, we have generalized the result of \textit{Chang-Fang-Zalcman} in a more compact form. As an application, we generalize the famous Br$\ddot{u}$ck conjecture \cite{Bru-1996}  with the idea of set sharing.
\end{abstract}
\section{Introduction Definitions and Results}
As we all know, Nevanlinna theory plays an important part in considering value distribution of meromorphic functions and non-trivial solutions of some complex differential equations. A function $ f $ is called meromorphic if it is analytic in the complex plane $ \mathbb{C} $ except at isolated poles. In what follows, we assume that the reader is familiar with the basic Nevanlinna Theory \cite{Hay-CP-1964,Yan-SV-1993}. It will be convenient to let $ E $ denote any set of positive real real numbers of finite linear measure, not necessarily the same at each occurrence.  Let $ f $ and $ g $ be two meromorphic functions having the same set of $ a $-points with the same multiplicities, we then say that $ f $ and $ g $ share the value $ a $ $ CM $ (counting multiplicities) and if we do not consider the multiplicities then $ f $ and $ g $ are said to share the value $ a $ $ IM $ means the poles of $ f $.\par  When $ a=\infty $, the zeros of $ f-a $ means the poles of $ f $.\par
\begin{defi}
	For a non-constant meromorphic function $f$ and any set $\mathcal{S}\subset\mathbb{\ol C}$, we define \beas E_{f}(\mathcal{S})=\displaystyle\bigcup_{a\in\mathcal{S}}\bigg\{(z,p)\in\mathbb{C}\times\mathbb{N}:f(z)=a,\;\text{with multiplicity}\; p\bigg\}, \eeas \beas\ol E_{f}(\mathcal{S})=\displaystyle\bigcup_{a\in\mathcal{S}}\bigg\{(z,1)\in\mathbb{C}\times\{1\}:f(z)=a\bigg\}.\eeas
\end{defi}
\par If $E_{f}(\mathcal{S})=E_{g}(\mathcal{S})$ ($\ol E_{f}(\mathcal{S})=\ol E_{g}(\mathcal{S})$) then we simply say $f$ and $g$ share $\mathcal{S}$ Counting Multiplicities(CM) (Ignoring Multiplicities(IM)).\par
Evidently, if $\mathcal{S}$ contains one element only, then it coincides with the usual definition of $CM (IM)$ sharing of values. \par 

In 1926, Nevanlinna first showed that a non-constant meromorphic function on the complex plane $\mathbb{C}$ is uniquely determined by the pre-images, ignoring multiplicities, of $5$ distinct values (including infinity). A few years latter, he showed that when multiplicities
are taken into consideration, $4$ points are enough and in that case either the two functions coincides or one is the bilinear transformation of the other one.\par  Recall that the spherical derivative of a meromorphic function $ f $ on a plane domain is \beas  f^{\#}(z)=\frac{|f^{\prime}(z)|}{1+|f(z)|^2}.  \eeas 
\par The sharing value problem between an entire functions and their derivatives was first studied by \textit{Rubel-Yang} \cite{Rub & Yan-1976} where they proved that if a non-constant entire function $ f $ and $ f^{\prime} $ share two distinct finite numbers $ a $, $ b $ $ CM $, then $ f\equiv f^{\prime} .$\par In $ 1979 $, \textit{Mues-Steinmetz} \cite{Mue & Ste-1979} improved the above theorem in the following manner.
\begin{theoA}\cite{Mue & Ste-1979}
	Let $ f $ be a non-constant entire function. If $ f $ and $ f^{\prime} $ share two distinct values $ a $, $ b $ $ IM $ then $ f\equiv f^{\prime} $.
\end{theoA}
\par We next recall the following well known definition of set sharing.\par Let $ S $ be a set of complex numbers and $ E_f(S)=\displaystyle\bigcup_{a\in S}\{z:f(z)=a\}, $ where each zero is counted according to its multiplicity. If we do not count the multiplicity, then the set $\bigcup_{a\in S}\{z:f(z)=a\}$ is denoted by $ \ol E_f(S). $\par If $ E_f(S)=E_g(S) $ we say that $ f $ and $ g $ share the set $ S $ $ CM $. On the other hand $ \ol E_f(S)=\ol E_g(S) $, we say that $ f $ and $ g $ share the set $ S $ $ IM $. Evidently, if $ S $ contains only one element, then it coincides with the usual definition of $ CM $ (respectively, $ IM $) sharing of values.\par We see from the following example that results of Rubel-Yang or Mues-Steinmetz are not in general true when we consider the sharing of a set of two elements instead of values.
\begin{exm}
	Let $ S =\bigg\{\displaystyle\frac{a}{3}, \displaystyle\frac{2a}{3}\bigg\}$, where $ a(\neq 0) $ be any complex number. Let $ f(z)=e^{-z}+a, $ then $ E_f(S)=E_{f^{\prime}}(S) $ but $ f\not\equiv f^{\prime} $.
\end{exm}
\par So for the uniqueness of an entire function and its derivative sharing a set, the cardinality of the range set should be at least three.
\par In this regard in $ 2003 $, using the properties of Normal families, \textit{Fang-Zalcman} \cite{Fan & al-JMMA-2003} obtained the following result.
\begin{theoB}\cite{Fan & al-JMMA-2003}
	Let $ S=\{0, a, b\} $, where $ a, b $ are two non-zero distinct complex numbers satisfying $ a^2\neq b^2 $, $ a\neq 2b $, $ a^2-ab+b^2\neq 0 $. If for a non-constant entire function $ f $, $ E_f(S)=E_{f^{\prime}}(S) $, then $ f\equiv f^{\prime} $.
\end{theoB}
\par In order to generalize the range set in the above theorem, in $ 2007 $ \textit{Chang-Fang-Zalcman} \cite{Cha & Fan & al-ADM-2007} obtained the following result.
\begin{theoC}\cite{Cha & Fan & al-ADM-2007}
	Let $ f $ be a non-constant entire function and let $ S=\{a, b, c\} $, where $ a, b $ and $ c $ are distinct complex numbers. If $ E_f(S)=E_{f^{\prime}}(S) $, then either
	\begin{enumerate}
		\item[(1)] $ f(z)=\mathcal{C}e^{z}; $ or
		\item[(2)] $ f(z)=\mathcal{C}e^{-z}+\frac{2}{3}(a+b+c)$ and $ (2a-b-c)(2b-c-a)(2c-a-b)=0; $ or
		\item[(3)] $ f(z)=\mathcal{C}e^{\frac{-1\pm i\sqrt{3}z}{2}}+\frac{3\pm i\sqrt{3}}{6}(a+b+c) $ and $ a^2+b^2+c^2-ab-bc-ca=0 $,
	\end{enumerate}
where $ \mathcal{C} $ is a non-zero constant.
\end{theoC}

\par We see from the next example that, conclusion of \textit{Theorem C} ceases to be hold if $ CM $ shared set $ \mathcal{S} $ be replaced by $ IM $ shared set.
\begin{exm}\label{ex1.2}\cite{Cha & Fan & al-ADM-2007}
	Let $ S=\{-1, 0, 1\} $ and $ f(z)=\sin z $ or $ \cos z $. then it is clear that $ f $ and $ f^{\prime} $ share the set $ S $ $ IM $ and $ f $ takes none of the forms $ (1)-(3) $ in Theorem C.
\end{exm}
\begin{rem}
	In Example 1.2, one may consider $ k $-th derivative of $ f $ instead of first derivative, when $ k $ is any odd positive integer to get the same.
\end{rem}
\par From the above discussions, one may note that a non-constant entire function and its first derivative when share a set of arbitrary three finite complex numbers $ a, b $ and $ c $ counting multiplicities, then it is possible to find out some specific forms of the function $ f $.\par 
\begin{rem}
	\par We have found a little gap in the statement of \emph{Theorem C}. This is because of the fact that the authors \textit{Chang-Fang-Zalcman} have been used \emph{Lemma 2.2} to prove their result \emph{Theorem C} for the  first derivative of a function $f$ i.e., for $k=1$. So one may noticed the following points. 
	\begin{enumerate}
		\item[(i).] In the statement of the \emph{Theorem C}, the author should mention the line from \emph{Lemma 2.2} as \emph{``let $ f $ be a non-constant entire function having zeros of multiplicities $\geq 1$"}.
		\item[(ii).] Since function $f$ must have zeros, so it is natural that the possible form of the function should not be of the form $ f(z)=\mathcal{C}e^{z}$ as it has no zeros at all.  
	\end{enumerate}
\end{rem}
So the natural question arises as follows:
\begin{ques}\label{q1.1}
	Is it possible to extend Theorem C for $ k $-th derivative of $ f $ ?
\end{ques}
\par If the answer of Question \ref{q1.1} is affirmative, then one may ask the following question: 
\begin{ques}
	What will be the possible forms of the non-constant entire function $ f $ ?
\end{ques}
\par Since $f$ and $f^{(k)}$ share the set $\mathcal{S}=\{a,b,c\}$, so one may observe that among all the possible relationship between $f$ and $f^{(k)}$, clearly $f^{(k)}\equiv f$ is the obvious one. So before going to state our main results, we want to discuss on a natural quarry  \textit{What is the general solution of $f^{(k)}\equiv f$ ?} The natural answer is $ f(z)=\mathcal{L}_{\theta}(z)$ (see \cite{Aha & CKMS & 2018,Ban & Aha & RCDCMP & 2018}) where we defined $\mathcal{L}_{\theta}(z)$ as follows \bea\label{e1.111} \mathcal{L}_{\theta}(z)=c_0e^z+c_1e^{\theta z}+c_2e^{\theta^2 z }+\ldots+c_{k-1}e^{\theta^{k-1}z},\eea $c_i\in\mathbb{C}$ for $i\in\{0,1,2,\ldots,k-1\}$ with $c_{k-1}\neq 0$ and $\theta=\displaystyle\cos\left(\frac{2\pi}{k}\right)+i\sin\left(\frac{2\pi}{k}\right)$.\par 
\par Answering all the questions mentioned above is the main motivation of writing this paper. We have tried to take care of the points we have mentioned in \emph{Remark 1.2}. Following is the main result of this paper.  
\begin{theo}\label{th1}
		Let $ f $ be a non-constant entire function, having zeros of multiplicity $ \geq k $ and let $ S=\{a, b, c\} $, where $ a, b $ and $ c $ are distinct complex numbers. If $ E_f(S)=E_{f^{(k)}}(S) $, then $ f $ takes one of the following forms:
	\begin{enumerate}
		\item[(1)] $ f(z)=\mathcal{L}_{\theta}(\beta z) ,$ where $ \beta $ is a root of the equation $ z^k-1=0, $ 
		\item[(2)] $ f(z)=\mathcal{L}_{\theta}(\eta z)+\frac{2}{3}(a+b+c)$, where $ \eta $ is a root of the equation $ z^k+1=0 $ and $ (2a-b-c)(2b-c-a)(2c-a-b)=0, $
		\item[(3)] $ f(z)=\mathcal{L}_{\theta}(\zeta z)+\frac{3\pm i\sqrt{3}}{6}(a+b+c) $, where $ \zeta (\neq 1) $ is a root of the equation $ z^{3k}-1=0 $ and $ a^2+b^2+c^2-ab-bc-ca=0 $,
	\end{enumerate}
	where $ \mathcal{L}_{\theta}(z) $ is defined in \emph{(\ref{e1.111})}.
\end{theo}

\section{Some Lemmas} We begin our investigation with the following lemmas, which are essential to
prove our main results.
\begin{lem}\cite{Clu & Hay-CMH-1966}\label{lm1}
	The order of an entire function having bounded  spherical derivative on $ \mathcal{C} $ is at most $ 1 $.
\end{lem}
\begin{lem}\cite{Fan & al-JMMA-2003}\label{lm2}
	Let $ \mathcal{F} $ be a family of holomorphic functions in a domain $ D $. Let $ k $ be a positive integer. Let $ a, b $ and $ c $ be three distinct finite complex numbers and $ M $ a positive number. If, for any $ f\in\mathcal{F} $, the zeros of $ f $ are of multiplicity $ \geq k $ and $ |f^{(k)}(z)|\leq M $ whenever $ f(z)\in\{a, b, c\} $, then $ \mathcal{F} $ is normal in $ D $.
\end{lem}
\begin{lem}\cite{Gun-JLMS-1998}\label{lm3}
	Let $ f $ be a  non-constant meromorphic function of finite order $ \rho $, and $ \epsilon >0 $ a constant. Then there exists a set $ E\subset [0,2\pi) $ which has linear measure zero, such that if $ \psi_0\in[0,2\pi) - E$, then there is a constant $ R_0=R_0(\psi_0) >0$ such that for all $ z $ satisfying $ arg z=\psi_0 $ and $ |z|>R_0 $, we have \beas  \bigg|\frac{f^{(k)}(z)}{f(z)}\bigg|\leq |z|^{k(\rho+\epsilon-1)}.  \eeas
\end{lem}
\begin{lem}\label{lm4}
	Let $ f $ be an entire function, and suppose that $ |f^{(k)}(z)| $ is unbounded on some ray $ \arg z=\theta $. Then there exists an infinite sequence of points $ z_n=r_ne^{\theta} $ where $ r_n\rightarrow\infty $, such that $ f^{(k)}(z_n)\rightarrow\infty $ and \bea\label{e2.1} \bigg|\frac{f(z_n)}{f^{(k)}(z_n)}\bigg|\leq (1+o(1))|z_n|^k  \eea as $ z_n\rightarrow\infty. $
\end{lem}
\begin{proof}
	Let $ T(r,f^{(k)},\theta)=\displaystyle\max_{\substack{0\leq |z|\leq r \\ \arg z=\theta}}\bigg\{\bigg|f^{(k)}(z)\bigg|\bigg\}$. It implies that there exists an infinite sequence of points $ z_n=r_ne^{i\theta} $ where $ r_n\rightarrow\infty $ such that $ T(r,f^{(k)},\theta)=|f^{(k)}(r_ne^{i\theta})| $ for all $ n $. Therefore for each $ n $, one can get the following easily \beas f(z_n)=f(0)+\int_{0}^{z_n}f^{\prime}(z)dz,  \eeas  \beas  f(z_n)=f(0)+\frac{z_n}{(1)!}\int_{0}^{z_n}\int_{0}^{z}f^{\prime\prime}(z)dzdz, \eeas \beas\vdots\eeas
	\beas f(z_n)=\sum_{i=0}^{k-1}\frac{z_n^i}{(i)!}f^{(i)}(0)+\bigg\{\int_{0}^{z_n}\overbrace{\int_{0}^{z}\ldots\int_{0}^{z}}^{(k-1)-\text{times}}f^{(k)}(z)\overbrace{dz\ldots dz}^{k-\text{times}}\bigg\}. \eeas So, applying triangle inequality, we get \beas && |f(z)|\\ &\leq&\bigg|\sum_{i=0}^{k-1}\frac{z_n^i}{(i)!}f^{(i)}(0)\bigg|+\bigg|\bigg\{\int_{0}^{z_n}\overbrace{\int_{0}^{z}\ldots\int_{0}^{z}}^{(k-1)-\text{times}}f^{(k)}(z)\overbrace{dz\ldots dz}^{k-\text{times}}\bigg\}\bigg|\\ &\leq& \bigg|\sum_{i=0}^{k-1}\frac{z_n^i}{(i)!}f^{(i)}(0)\bigg|+\bigg|f^{(k)}(z)\bigg|\bigg|\bigg\{\int_{0}^{z_n}\overbrace{\int_{0}^{z}\ldots\int_{0}^{z}}^{(k-1)-\text{times}}\overbrace{dz\ldots dz}^{k-\text{times}}\bigg\}\bigg|\\ &\leq& \bigg|\sum_{i=0}^{k-1}\frac{z_n^i}{(i)!}f^{(i)}(0)\bigg|+\bigg|f^{(k)}(z)\bigg||z_n|^k.\eeas Since $ f^{(k)}(z)\rightarrow\infty $, so we obtained (\ref{e2.1}).
\end{proof}
\begin{lem}\cite{Mar-AFSUT-1931}\label{lm5}
	A class $ \mathcal{C} $ of functions $ f $ meromorphic in a domain $ D\subset\mathcal{C} $ is normal in $ D $ if and only if $ f^{\#} $ is uniformly bounded on any compact subset of $ D $ for $ f\in\mathcal{C}. $
\end{lem}
\begin{lem}\cite{Hei & Kor & Rat-BLMS-2004,Ngo & Ost-ANASD-272}\label{lm6}
	Let $ f $ be an entire function of order at most $ 1 $ and $ k $ be a positive integer. Then \beas m\left(r,\frac{f^{(k)}}{f}\right)=o(\log r),\;\;\;\;\text{as}\;r\rightarrow\infty.  \eeas
\end{lem}
\begin{lem}\label{lm7}
	Let $ \alpha $ be a non-constant entire function and $ a, b $ and $ c $ are three distinct finite complex numbers. Then there does not exist an entire function $ f $ satisfying the differential equation \bea\label{e2.2} \frac{\left(f^{(k)}-a\right)\left(f^{(k)}-b\right)\left(f^{(k)}-c\right)}{(f-a)(f-b)(f-c)}=e^{\alpha}.  \eea
\end{lem}
\begin{proof}
	Let if possible there exists an entire function satisfying (\ref{e2.2}). Then we see that $ |f^{(i)}(z)|\leq \max\{a, b, c\} $ whenever $ f(z)\in\{a, b, c\} $, $i\in\{1, 2,..., k\}$. Thus by Lemma \ref{lm2}, the family $ \mathcal{F}_w=\{f_w:w\in\mathbb{C}\}, $ where $ f_w(z)=f(w+z) $ is normal on the unit disc, so by Marty's Theorem, we get $ f^{\#}(w)=(f_w)^{\#}(0) $ is uniformly bounded for all $ w\in\mathbb{C} $. Therefore from Lemma \ref{lm1}, we get that $ f $ has order at most $ 1 $. \par Now from (\ref{e2.2}), we obtained $ \alpha(z)=\mathcal{A}z+\mathcal{B}, $ where $ \mathcal{A} $ and $ \mathcal{B} $ are two constants. It is clear that $ \mathcal{A}\neq 0 $, since $ \alpha $ is non-constant.\par Next we claim that $ abc\neq 0 $. On contrary, let $ abc=0. $ i.e., $ a=0 $ or $ b=0 $ or $ c=0 $. Without any loss of generality, we may assume that $ a=0 $. Then from (\ref{e2.2}), we get \beas   \frac{f^{(k)}\left(f^{(k)}-b\right)\left(f^{(k)}-c\right)}{f(f-b)(f-c)}=e^{\mathcal{A}z+\mathcal{B}}. \eeas Again we see that \beas &&  \frac{f^{(k)}\left(f^{(k)}-b\right)\left(f^{(k)}-c\right)}{f(f-b)(f-c)}\\ &=&\frac{\left(f^{(k)}\right)^3}{f(f-b)(f-c)}-\frac{(b+c)\left(f^{(k)}\right)^2}{f(f-b)(f-c)}+\frac{bcf^{(k)}}{f(f-b)(f-c)}\\&=&\frac{f^{(k)}}{f}\frac{f^{(k)}}{f-b}\frac{f^{(k)}}{f-c}-\frac{b+c}{b-c}\left(\frac{f^{(k)}}{f-b}-\frac{f^{(k)}}{f-c}\right)+bc\left(\frac{\mathcal{A}_1f^{(k)}}{f}+\frac{\mathcal{B}_1f^{(k)}}{f-b}+\frac{\mathcal{B}_1f^{(k)}}{f-c}\right), \eeas where $ \mathcal{A}_1, \mathcal{B}_1 $ and $ \mathcal{C}_1 $ are constants. Next we see that there exists $ \mathcal{A}_2, \mathcal{B}_2 $ and $ \mathcal{C}_2 $ (\cite{bibid}) such that \beas && m\left(r,\frac{f^{(k)}(f^{(k)}-b)(f^{(k)}-c)}{f(f-b)(f-c)}\right)\\ &\leq&\mathcal{A}_2\;m\left(r,\frac{f^{(k)}}{f}\right)+\mathcal{B}_2\;m\left(r,\frac{f^{(k)}}{f-b}\right)+\mathcal{B}_2\;m\left(r,\frac{f^{(k)}}{f-c}\right)+O(1). \eeas Thus by Lemma \ref{lm6}, we get \beas T(r,e^{\mathcal{A}z+\mathcal{B}})=m(r,e^{\mathcal{A}z+\mathcal{B}})=o(\log r),  \eeas which is not possible since $ \mathcal{A}\neq 0. $\par Therefore $ abc\neq 0 $. Next we get \bea \label{e2.3} g(z)=f(z/\mathcal{A})\;\; i.e.,\;\; g^{(k)}(z)=\frac{1}{\mathcal{A}^k}f^{(k)}(z/\mathcal{A}).  \eea Using (\ref{e2.3}) in (\ref{e2.2}), we get \bea\label{e2.4}\frac{\left(g^{(k)}-a/\mathcal{A}^k\right)\left(g^{(k)}-b/\mathcal{A}^k\right)\left(g^{(k)}-c/\mathcal{A}^k\right)}{(g-a)(g-b)(g-c)}\equiv\mathcal{C}e^z, \eea where $ \mathcal{C}=\displaystyle\frac{e^{\mathcal{B}}}{\mathcal{A}^{3k}}\neq 0. $ Next (\ref{e2.4}) can be written as follows \bea\label{e2.5} \frac{\left(g^{(k)}\right)^3+\mathcal{C}_1\left(g^{(k)}\right)^2+\mathcal{C}_2g^{(k)}}{(g-a)(g-b)(g-c)}-\mathcal{C}e^z =\frac{\mathcal{C}_3}{(g-a)(g-b)(g-c)}, \eea where $ \mathcal{C}_j $ are constants with $ \mathcal{C}_3\neq 0 $. With $ \epsilon=\frac{1}{3} $, Lemma \ref{lm3} shows that there exists a set $ E\subset [0,2\pi) $ of measure zero such that for each $ \psi_0 \in [0,2\pi)-E $, there is a constant $ R_0=R_0(\psi_0)>0 $ such that whenever $ arg z=\psi_0 $ and $ |z|>R_0 $, \bea \label{e2.6} \bigg|\frac{\left(g^{(k)}\right)^3+\mathcal{C}_1\left(g^{(k)}\right)^2+\mathcal{C}_2g^{(k)}}{(g-a)(g-b)(g-c)}\bigg|\leq K|z|, \eea for some positive constant $ K $. Now we may suppose that $ \pi/2 $ and $ 3\pi/2 $ are continued in the set $ E $. Then $ [0,2\pi)-E=E_1\cup E_2$, where $ E_1=\{\theta\in[0,2\pi):\cos\theta >0\} $ and $ E_2=\{\theta\in[0,2\pi):\cos\theta <0\} $. Let $ \theta\in E_1 $, then by (\ref{e2.5}) and (\ref{e2.6}), we have for sufficiently large $ r $, \beas && \bigg|\frac{\mathcal{C}_3}{\left(g(re^{i\theta})-a\right)\left(g(re^{i\theta})-b\right)\left(g(re^{i\theta})-c\right)}\bigg|\\ &=&\bigg|\frac{\left(g^{(k)}(re^{i\theta})\right)^3+\mathcal{C}_1\left(g^{(k)}(re^{i\theta})\right)^2+\mathcal{C}_2g^{(k)}(re^{i\theta})}{\left(g(re^{i\theta})-a\right)\left(g(re^{i\theta})-b\right)\left(g(re^{i\theta})-c\right)}-\mathcal{C}e^{re^{i\theta}}\bigg|\\ &\geq& |\mathcal{C}|e^{r\cos\theta}-Kr\\ &\rightarrow&\infty,\;\;\;\text{as}\;\;\; r\rightarrow\infty.  \eeas
	It follows that \bea\label{e2.7} g(re^{i\theta})\rightarrow a,\; b\;\; \text{or}\;\; c,\;\;\;\text{as}\;\;\; r\rightarrow\infty.  \eea \par Next let $ \theta\in E_2 $. We claim that $ |g^{(k)}(re^{i\theta})| $ is bounded as $ r\rightarrow\infty $. Suppose on the contrary that $ |g^{(k)}(re^{i\theta})| $ is unbounded as $ r\rightarrow\infty $. Then by Lemma \ref{lm4}, there exists a sequence $ r_n\rightarrow\infty $ such that $ |g^{(k)}(re^{i\theta})|\rightarrow\infty $ and \bea\label{e2.8}\bigg|g(re^{i\theta})g^{(k)}(re^{i\theta})\bigg|\leq (1+o(1))r_n^k.   \eea Now with $ |g^{(k)}(r_ne^{i\theta})|\rightarrow\infty $, we note that \bea\label{e2.9} \bigg|\frac{\left(g(r_ne^{i\theta})-a\right)\left(g(r_ne^{i\theta})-b\right)\left(g(r_ne^{i\theta})-c\right)}{\left(g^{(k)}(r_ne^{i\theta})\right)^3+\mathcal{C}_1\left(g^{(k)}(r_ne^{i\theta})\right)^2+\mathcal{C}_2g^{(k)}(r_ne^{i\theta})}\bigg|\leq (1+o(1))r_n^{3k} .\eea Again since $ |g^{(k)}(r_ne^{i\theta})|\rightarrow\infty $, it follows from (\ref{e2.5}) that \bea\label{e2.10} &&\bigg|\frac{\left(g(r_ne^{i\theta})-a\right)\left(g(r_ne^{i\theta})-b\right)\left(g(r_ne^{i\theta})-c\right)}{r_n^{3k}\mathcal{C}_3}\bigg|\\ &=& \bigg|\frac{\left(g^{(k)}(r_ne^{i\theta})\right)^3+\mathcal{C}_1\left(g^{(k)}(r_ne^{i\theta})\right)^2+\mathcal{C}_2g^{(k)}(r_ne^{i\theta})-\mathcal{C}_3}{r_n^{3k}|\mathcal{C}_3||\mathcal{C}|e^{r_ne^{i\theta}}}\bigg|\nonumber\\ &=&\nonumber r_n^{-3k}e^{-r_n\cos\theta}\frac{|\left(g^{(k)}(r_ne^{i\theta})\right)^3+\mathcal{C}_1\left(g^{(k)}(r_ne^{i\theta})\right)^2+\mathcal{C}_2g^{(k)}(r_ne^{i\theta})-\mathcal{C}_3|}{|\mathcal{C}_3\mathcal{C}|}\\ &\rightarrow& \infty\nonumber.  \eea Thus from (\ref{e2.5}), (\ref{e2.9}) and (\ref{e2.10}), we get \beas && 1-o(1)\\ &\leq& \bigg|r_n^{3k}\frac{\left(g^{(k)}(r_ne^{i\theta})\right)^3+\mathcal{C}_1\left(g^{(k)}(r_ne^{i\theta})\right)^2+\mathcal{C}_2g^{(k)}(r_ne^{i\theta})-\mathcal{C}_3}{\left(g(r_ne^{i\theta})-a\right)\left(g(r_ne^{i\theta})-b\right)\left(g(r_ne^{i\theta})-c\right)}\bigg|\\ &\leq &\bigg|\frac{r_n^{3k}\mathcal{C}_3}{\left(g(r_ne^{i\theta})-a\right)\left(g(r_ne^{i\theta})-b\right)\left(g(r_ne^{i\theta})-c\right)}\bigg|+ |\mathcal{C}|r_n^{3k}e^{r_n\cos\theta}\\ &\rightarrow& 0, \eeas which is absurd. Hence our suppositions that $ |g^{(k)}(r_ne^{i\theta})| $ is bounded as $ r\rightarrow\infty $ for each $ \theta\in E_2 $. One can get easily that \beas g(re^{i\theta}) = \sum_{i=0}^{k-1}\frac{z_n^i}{(i)!}f^{(i)}(0)+\left(e^{i\theta}\right)^k\bigg\{\int_{0}^{r}\overbrace{\int_{0}^{t}\ldots\int_{0}^{t}}^{(k-1)-\text{times}}f^{(k)}(z)\overbrace{dt\ldots dt}^{k-\text{times}}\bigg\}.\eeas So \bea\label{e2.11} && |g(re^{i\theta})|\\ &\leq& \bigg|\sum_{i=0}^{k-1}\frac{z_n^i}{(i)!}f^{(i)}(0)\bigg|+\bigg|\bigg\{\int_{0}^{r}\overbrace{\int_{0}^{t}\ldots\int_{0}^{t}}^{(k-1)-\text{times}}f^{(k)}(z)\overbrace{dt\ldots dt}^{k-\text{times}}\bigg\}\bigg|\nonumber\\ &\leq&\bigg|\sum_{i=0}^{k-1}\frac{z_n^i}{(i)!}f^{(i)}(0)\bigg|+\mathcal{M}r^{k}\nonumber,   \eea where $ \mathcal{M}=\mathcal{M}(\theta) $ is a positive constant depending on $ \theta $.\par Hence by (\ref{e2.7}) and (\ref{e2.11}), for every $ \theta\in [0,2\pi)-E $, there exists a positive constant $ \mathcal{L}=\mathcal{L}(\theta) $ such that for $ z=re^{i\theta} $ with $ r>r_0, $ \bea\label{e2.12} \bigg|\frac{g(z)}{z^k}\bigg|\leq\mathcal{L}. \eea Since $ g $ has order at most $ 1 $, it follows from (\ref{e2.7}), (\ref{e2.12}), the Phrag$\acute{e}$n-Lindel$\ddot{o}$f Theorem \cite{bibid}, and by Lioville's Theorem, $ g $ is a polynomial of degree at most $ k $, which is impossible by (\ref{e2.4}). This completes the proof.
	
\end{proof}
\begin{lem}\label{lm8}
	Let $ S=\{a, b, c\} $ where $ a, \; b$ and $ c $ be any three distinct finite complex numbers and $ \mathcal{A} $ a non-zero constant. If $ E_f(S)=E_{f^{(k)}}(S), $ where $ f $ is an entire function having zeros of multiplicities $\geq k$ and satisfying $ f^{(k)}\not\equiv 0 $ and \bea\label{e2.13} \frac{\left(f^{(k)}-a\right)\left(f^{(k)}-b\right)\left(f^{(k)}-c\right)}{(f-a)(f-b)(f-c)}\equiv\mathcal{A},\eea then $ f $ must take one of the following forms:
	\begin{enumerate}
		\item[(1)] $ f(z)=\mathcal{L}_{\theta}(\beta z) ,$ where $ \beta $ is a root of the equation $ z^k-1=0, $ 
	\item[(2)] $ f(z)=\mathcal{L}_{\theta}(\eta z)+\frac{2}{3}(a+b+c)$, where $ \eta $ is a root of the equation $ z^k+1=0 $ and $ (2a-b-c)(2b-c-a)(2c-a-b)=0, $
	\item[(3)] $ f(z)=\mathcal{L}_{\theta}(\zeta z)+\frac{3\pm i\sqrt{3}}{6}(a+b+c) $, where $ \zeta (\neq 1) $ is a root of the equation $ z^{3k}-1=0 $ and $ a^2+b^2+c^2-ab-bc-ca=0 $,
	\end{enumerate}
	where $ \mathcal{C} $ is a non-zero constant.
\end{lem}
\begin{proof}
	From the proof of Lemma \ref{lm7}, we note that $ f $ has order at most $ 1 $. Since $ f $ and $ f^{(k)} $ have the same order and $f$ having zeros of multiplicities $\geq k$ satisfying  $ f^{(k)}\not\equiv 0 $ and $ E_f(S)=E_{f^{(k)}}(S)$, so one must have the following form \bea\label{e2.14} f^{(k)}(z)=c_0\alpha^k e^{\alpha z}+c_1\alpha^k e^{\alpha\theta z}+\ldots+c_{k-1}\alpha^k e^{\alpha\theta^{k-1} z}=\alpha^k\mathcal{L}_{\theta}(\alpha  z),\text{(say)}  \eea where $ c_{i}\in\mathbb{C}$, for $i\in\{0,1,2,\ldots,k-1\}$ with $c_{k-1}\neq 0$, $ \alpha\in\mathbb{C}-\{0\} $,  \beas \theta=\cos\left(\frac{2\pi}{k}\right)+i\sin\left(\frac{2\pi}{k}\right)\eeas and \beas \mathcal{L}_{\theta}(\alpha z)=c_0 e^{\alpha z}+c_1 e^{\alpha\theta z}+\ldots+c_{k-1} e^{\alpha\theta^{k-1} z}.\eeas On integrating (\ref{e2.14}) $ k $-times, we get \bea\label{e2.15} f(z)=\mathcal{L}_{\theta}(\alpha z)+\mathcal{Q}_{k-1}(z), \eea where $ \mathcal{Q}_{k-1} $ is a polynomial of degree $ \leq k-1. $\par Next using (\ref{e2.14}) and (\ref{e2.15}), we get from (\ref{e2.13})
	
	\bea\label{e2.16} && (\alpha^{3k}-\mathcal{A})\left(\mathcal{L}_{\theta}(\alpha z)\right)^3+\left(\mathcal{L}_1\alpha^{2k}-3\mathcal{A}\mathcal{Q}_{k-1}-\mathcal{A}\mathcal{L}_1\right)\left(\mathcal{L}_{\theta}(\alpha z)\right)^2\\ &&+\nonumber \left(\mathcal{L}_2\alpha^k-3\mathcal{A}\mathcal{Q}_{k-1}^2-2\mathcal{A}\mathcal{L}_1\mathcal{Q}_{k-1}-\mathcal{A}\mathcal{L}_2\right)\mathcal{L}_{\theta}(\alpha z)\\ &&\nonumber +\left(\mathcal{L}_3-\mathcal{A}\mathcal{Q}_{k-1}^3-\mathcal{A}\mathcal{L}_1\mathcal{Q}_{k-1}^2-\mathcal{A}\mathcal{L}_2\mathcal{Q}_{k-1}-\mathcal{A}\mathcal{L}_3\right)\equiv 0,    \eea where $ \mathcal{L}_1=-(a+b+c) $, $ \mathcal{L}_2=ab+bc+ca $ and $ \mathcal{L}_3=-abc $. It follows that \bea\label{e2.17} \alpha^{3k}=\mathcal{A},  \eea \bea\label{e2.18} \mathcal{L}_1\alpha^{2k}=\mathcal{A}(3\mathcal{Q}_{k-1}+\mathcal{L}_1), \eea  \bea\label{e2.19} \mathcal{L}_2\alpha^k=\mathcal{A}(3\mathcal{Q}_{k-1}^2+2\mathcal{L}_1\mathcal{Q}_{k-1}+\mathcal{L}_2), \eea  \bea\label{e2.20} \mathcal{L}_3=\mathcal{A}(\mathcal{Q}_{k-1}^3+\mathcal{L}_1\mathcal{Q}_{k-1}^2+\mathcal{L}_2\mathcal{Q}_{k-1}+\mathcal{L}_3).\eea We now discuss the following different cases.\\
	\noindent{\bfseries{Case 1.}} Let $ \alpha\in\{z:z^k-1=0\} $. Then from (\ref{e2.17}) and (\ref{e2.18}), we get $ \mathcal{A}=1 $ and $ \mathcal{Q}_{k-1}=0 $. Thus we see that \beas f(z)=\mathcal{L}_{\theta}(\beta z),\eeas where $ \beta  $ is a root of the equation $ z^k-1=0 .$\\
	\noindent{\bfseries{Case 2.}} Let $ \alpha \in\{z:z^k+1=0\} $. Then from (\ref{e2.17}) and (\ref{e2.18}), we see that $ \mathcal{A}=-1 $ and $ \mathcal{Q}_{k-1}=-\frac{2}{3}\mathcal{L}_1. $  It follows from (\ref{e2.20}) that \beas 2\mathcal{L}_1^3-9\mathcal{L}_1\mathcal{L}_2+27\mathcal{L}_3=0 \eeas which in turn implies that \beas (2a-b-c)(2b-c-a)(2c-a-b)=0. \eeas \par In this case, we get \beas f(z)=\mathcal{L}_{\theta}(\eta z)+\frac{2}{3}(a+b+c),  \eeas where $ \eta $ is a root of the equation $ z^k+1=0. $\\
	\noindent{\bfseries{Case 3.}} Let $ \alpha\not\in\{z:z^k-1=0\}\cup\{z:z^k+1=0\} $. Then by (\ref{e2.17}) and (\ref{e2.18}), we get \bea \label{e2.21} \mathcal{Q}_{k-1}=\frac{1-\alpha^k}{3\alpha^k}\mathcal{L}_1. \eea Then by (\ref{e2.17}), (\ref{e2.19}) and (\ref{e2.21}), we get \bea\label{e2.22}\mathcal{L}_2=\frac{(\mathcal{L}_1)^2}{3}.    \eea Next by (\ref{e2.17}), (\ref{e2.20}), (\ref{e2.21}) and (\ref{e2.22}), we get \bea\label{e2.23} (1-\alpha^{3k})\mathcal{L}_3=\frac{1}{27}(1-\alpha^{3k})\mathcal{L}_1^3. \eea
	\noindent{\bfseries{Subcase 3.1.}} If $ \alpha^{3k}\neq 1 $, then $ \mathcal{L}_3=(\mathcal{L}_1)^3/27 $. This with  (\ref{e2.22}) shows that $ a=b=c $, which is not possible.\\
	\noindent{\bfseries{Subcase 3.2.}} Hence $ \alpha^{3k}-1=0 $. i.e., $ \alpha^{k}=\displaystyle\frac{-1\pm i\sqrt{3}}{2} $. Thus we have $ \mathcal{Q}_{k-1}=-\displaystyle\frac{3\pm i\sqrt{3}}{6}\mathcal{L}_1. $ After simplifying (\ref{e2.22}), we get $ a^2+b^2+c^2-ab-bc-ca=0. $ Therefore we see that \beas f(z)=\mathcal{L}_{\theta}(\zeta z)+\frac{3\pm i\sqrt{3}}{6}(a+b+c), \eeas where $ \zeta(\neq 1) $ is a root of the equation $ z^{3k}-1=0. $
\end{proof}
\section{Proof of Theorem \ref{th1}} Since $ E_f(S)=E_f^{(k)}(S) $, therefore it is clear that \bea\label{e3.1} \frac{\left(f^{(k)}-a\right)\left(f^{(k)}-b\right)\left(f^{(k)}-c\right)}{(f-a)(f-b)(f-c)}\equiv e^{\alpha (z)},\eea where $ \alpha $ is an entire function. We note that by Lemma \ref{lm7}, $ \alpha $ is constant. Then we set $ \mathcal{A}=e^{\alpha} .$ Thus from (\ref{e3.1}) changes to \bea\label{e3.2} \frac{\left(f^{(k)}-a\right)\left(f^{(k)}-b\right)\left(f^{(k)}-c\right)}{(f-a)(f-b)(f-c)}\equiv\mathcal{A}. \eea Next we discuss the following cases.\\
\noindent{\bfseries{Case 1.}} If $ f^{(k)}\neq 0 $, then by Lemma \ref{lm8}, we see that $ f $ takes one of the three forms (1)-(3). So we are done.\\
\noindent{\bfseries{Case 2.}} If $ f^{(k)} $ vanishes at some point $ z_0\in\mathbb{C} $. i.e., $ f^{(k)}(z_0)=0 $. Now differentiating both sides of (\ref{e3.2}), we get  \bea\label{e3.3} &&\bigg\{3\left(f^{(k)}\right)^2-2(a+b+c)f^{(k)}+(ab+bc+ca)\bigg\}f^{(k+1)}\\ &\equiv& \mathcal{A}\bigg\{3f^2-2(a+b+c)f+(ab+bc+ca)\bigg\}f^{\prime}\nonumber.   \eea \par Let $ f^{(k)}(z_0)=0 $, $ k\leq n$. So we may assume \beas f(z)=f(z_0)+A_n(z-z_0)^n+\ldots \eeas \par It is clear that $ f^{(k)}(z)=B_n(z-z_0)^{n-k}+\ldots $ and $ f^{\prime}(z)=nA(z-z_0)^{n-1}+\ldots $. We see that $ L. H. S $ of (\ref{e3.3}) vanishes at $ z_0 $ to order $ n-k $ while $ R. H. S $ of (\ref{e3.3}) vanishes to the order at least $ n-1 $, which is not possible. 
\section{Some Application}
In $ 1996 $, the following conjecture was proposed by Br\"{u}ck \cite{Bru-1996}.\par 
\begin{conj}\cite{Bru-1996}
	Let $ f $ be a non-constant entire function. Suppose that $ \rho_1(f) $ is not a positive integer or infinite. If $ f $ and $ f^{\prime} $ share one finite value $ a $ $ CM $, then  \beas  \frac{f^{\prime}-a}{f-a}=c,\eeas for some non-zero constant $ c $, where $ \rho_1(f) $ is the first iterated order of $ f $ which is defined by \beas  \rho_1(f)=\limsup_{r\rightarrow\infty}\frac{\log\log T(r,f)}{\log r}.\eeas 
	
\end{conj}
\par Many authors (for the case of differences see \cite{Hei & Kor & JMMA & 2009,Hua & Zha & AM & 2018} and for the cases of derivatives or differential polynomials  see \cite{Aha & Ban & BTUB & 2017} - \cite{Ban & Aha & JCA & 2019} and \cite{Che & Sho & TJM & 2004,Che & Sho & JKMS & 2005,Gun & Yan & JMMA & 1998})  have studied the conjecture under some additional conditions. But the main conjecture is still open. In this direction, it is interesting to ask the following two questions.
\begin{ques}
	Does the conjecture hold if one considers a set having three arbitrary finite complex numbers instead of a value ?
\end{ques}
\begin{ques}
	Is it possible to replace first derivative $ f^{(\prime)} $ by a more general derivative $ f^{(k)} ?$
\end{ques}
\begin{rem}
	Note that Lemma \ref{lm8} answers the above questions in some sense.
\end{rem}

\end{document}